\documentclass[12pt,reqno]{amsart}

\usepackage{color} 
\usepackage{amstext} \usepackage{amsthm} \usepackage{amsmath} \usepackage{amssymb}
\usepackage{latexsym} \usepackage{amsfonts} \usepackage{graphicx} \usepackage{texdraw} \usepackage{graphpap}
\usepackage{enumerate}

\usepackage[pagebackref,hypertexnames=false, colorlinks, citecolor=red, linkcolor=red]{hyperref}
\usepackage[backrefs]{amsrefs}

\usepackage[inline,nomargin]{fixme}
\DeclareMathOperator{\WOT}{WOT}

\input txdtools

\begin{document} 
\newcommand{\ci}[1]{_{ {}_{\scriptstyle #1}}}

\newcommand{\norm}[1]{\ensuremath{\left\|#1\right\|}} \newcommand{\abs}[1]{\ensuremath{\left\vert#1\right\vert}}
\newcommand{\ip}[2]{\ensuremath{\left\langle#1,#2\right\rangle}} \newcommand{\p}{\ensuremath{\partial}}
\newcommand{\pr}{\mathcal{P}}

\newcommand{\pbar}{\ensuremath{\bar{\partial}}} \newcommand{\db}{\overline\partial} \newcommand{\D}{\mathbb{D}}
\newcommand{\B}{\mathbb{B}} \newcommand{\Sn}{{\mathbb{S}_n}} \newcommand{\T}{\mathbb{T}} \newcommand{\R}{\mathbb{R}}
\newcommand{\Z}{\mathbb{Z}} \newcommand{\C}{\mathbb{C}} \newcommand{\N}{\mathbb{N}} \newcommand{\scrH}{\mathcal{H}}
\newcommand{\scrL}{\mathcal{L}} \newcommand{\td}{\widetilde\Delta}

\newcommand{\Aa}{\mathcal{A}} \newcommand{\BB}{\mathcal{B}} \newcommand{\HH}{\mathcal{H}} \newcommand{\KK}{\mathcal{K}}
\newcommand{\LL}{\mathcal{L}} \newcommand{\MM}{\mathcal{M}} \newcommand{\FF}{\mathcal{F}}

\newcommand{\AapBerg}{\Aa_p(\B_n)} \newcommand{\AaFock}{\Aa_\phi (\C^n)} \newcommand{\AatwoBerg}{\Aa_2(\B_n)}
\newcommand{\Om}{\Omega} \newcommand{\La}{\Lambda} \newcommand{\AaordFock}{\Aa (\C^n)}

\newcommand{\rk}{\operatorname{rk}} \newcommand{\card}{\operatorname{card}} \newcommand{\ran}{\operatorname{Ran}}
\newcommand{\osc}{\operatorname{OSC}} \newcommand{\im}{\operatorname{Im}} \newcommand{\re}{\operatorname{Re}}
\newcommand{\tr}{\operatorname{tr}} \newcommand{\vf}{\varphi} \newcommand{\f}[2]{\ensuremath{\frac{#1}{#2}}}

\newcommand{\kzp}{k_z^{(p,\alpha)}} \newcommand{\klp}{k_{\lambda_i}^{(p,\alpha)}} \newcommand{\TTp}{\mathcal{T}_p}

\newcommand{\vp}{\varphi} \newcommand{\al}{\alpha} \newcommand{\be}{\beta} \newcommand{\la}{\lambda}
\newcommand{\li}{\lambda_i} \newcommand{\lb}{\lambda_{\beta}} \newcommand{\Bo}{\mathcal{B}(\Omega)}
\newcommand{\Bbp}{\mathcal{B}_{\beta}^{p}} \newcommand{\Bbt}{\mathcal{B}(\Omega)} \newcommand{\Lbt}{L_{\beta}^{2}}
\newcommand{\Kz}{K_z} \newcommand{\kz}{k_z} \newcommand{\Kl}{K_{\lambda_i}} \newcommand{\kl}{k_{\lambda_i}}
\newcommand{\Kw}{K_w} \newcommand{\kw}{k_w} \newcommand{\Kbz}{K_z} \newcommand{\Kbl}{K_{\lambda_i}}
\newcommand{\kbz}{k_z} \newcommand{\kbl}{k_{\lambda_i}} \newcommand{\Kbw}{K_w} \newcommand{\kbw}{k_w}
\newcommand{\BL}{\mathcal{L}\left(\mathcal{B}(\Omega), L^2(\Om;d\sigma)\right)}
\newcommand{\Fpphi}{\ensuremath{{\mathcal{F}}_\phi ^p }}
\newcommand{\Ftwophi}{\ensuremath{{\mathcal{F}}_\phi ^2 }}
\newcommand{\incn}{\ensuremath{\int_{\C}}}
\newcommand{\Finfphi}{\ensuremath{\mathcal{F}_\phi ^\infty }}
\newcommand{\Fp}{\ensuremath{\mathcal{F} ^p }} \newcommand{\Fq}{\ensuremath{\mathcal{F} ^q }}
\newcommand{\Ft}{\ensuremath{\mathcal{F} ^2 }} \newcommand{\Lt}{\ensuremath{L ^2 }}
\newcommand{\Lp}{\ensuremath{L ^p }}
\newcommand{\Fonephi}{\ensuremath{\mathcal{F}_\phi ^1 }}
\newcommand{\Lpphi}{\ensuremath{L_\phi ^p}}
\newcommand{\Ltwophi}{\ensuremath{L_\phi ^2}}
\newcommand{\Lonephi}{\ensuremath{L_\phi ^1}}
\newcommand{\af}{\mathfrak{a}} \newcommand{\bb}{\mathfrak{b}} \newcommand{\cc}{\mathfrak{c}}
\newcommand{\Fqphi}{\ensuremath{\mathcal{F}_\phi ^q }}

\newcommand{\entrylabel}[1]{\mbox{#1}\hfill}

\newenvironment{entry} {\begin{list}{X}%
  {\renewcommand{\makelabel}{\entrylabel}%
      \setlength{\labelwidth}{55pt}%
      \setlength{\leftmargin}{\labelwidth}
      \addtolength{\leftmargin}{\labelsep}%
   }%
}


\numberwithin{equation}{section}

\newtheorem{thm}{Theorem}[section] \newtheorem{lm}[thm]{Lemma} \newtheorem{cor}[thm]{Corollary}
\newtheorem{conj}[thm]{Conjecture} \newtheorem{prob}[thm]{Problem} \newtheorem{prop}[thm]{Proposition}
\newtheorem*{prop*}{Proposition}

\theoremstyle{remark} \newtheorem{rem}[thm]{Remark} \newtheorem*{rem*}{Remark} \newtheorem{example}[thm]{Example}

\theoremstyle{definition} \newtheorem{definition}[thm]{Definition}

\title{Localization and Compactness in Bergman and Fock Spaces}

\author[J. Isralowitz]{Joshua Isralowitz} \address{Joshua Isralowitz, Department of Mathematics and Statistics \\ University at Albany \\ 1400 Washington Ave. \\ Albany, NY USA 12222} \email{Jisralowitz@albany.edu} \urladdr{http://www.albany.edu/~ji126652/}

\author[M. Mitkovski]{Mishko Mitkovski$^\dagger$} \address{Mishko Mitkovski, Department of Mathematical Sciences\\
Clemson University\\ O-110 Martin Hall, Box 340975\\ Clemson, SC USA 29634} \email{mmitkov@clemson.edu}
\urladdr{http://people.clemson.edu/~mmitkov/} \thanks{$\dagger$ Research supported in part by National Science Foundation
DMS grant \# 1101251.}

\author[B. D. Wick]{Brett D. Wick$^\ddagger$} \address{Brett D. Wick, School of Mathematics\\ Georgia Institute of
Technology\\ 686 Cherry Street\\ Atlanta, GA USA 30332-0160} \email{wick@math.gatech.edu}
\urladdr{www.math.gatech.edu/~wick} \thanks{$\ddagger$ Research supported in part by National Science Foundation DMS
grants \# 1001098 and \# 0955432.}


\begin{abstract} In this paper we study the compactness of operators on the Bergman space of the unit ball and
on very generally weighted Bargmann-Fock spaces in terms of the behavior of their Berezin transforms and the norms of the operators acting on reproducing kernels. In particular, in the Bergman space setting we show how a vanishing
Berezin transform combined with certain (integral) growth conditions on an operator $T$ are sufficient to imply that the
operator is compact.  In the weighted Bargmann-Fock space setting we show that the reproducing kernel thesis for compactness holds for operators satisfying similar growth conditions.  The main results extend the results of Xia and Zheng to the case of the
Bergman space when $1 < p < \infty$, and in the weighted Bargmann-Fock space setting, our results provide new, more general conditions that imply the work of Xia and Zheng via a more familiar approach that can also handle the $1 < p < \infty$ case. \end{abstract}

\maketitle

\section{Introduction} \label{Intro}

The Bargmann-Fock space $\mathcal{F}^p:=\mathcal{F}^p(\C^n)$ is the collection of entire functions $f$ on $\C^n$ such that $f(\cdot) e^{- \frac{\abs{\cdot}}{2}} \in L^p(\C^n, dv)$.  It is well known that $\FF^2$ is a reproducing kernel Hilbert
space with reproducing kernel given by $K_z(w)=e^{\overline{z}w}$.  As usual, we denote by $k_z$ the normalized
reproducing kernel at $z$. For a bounded operator $T$ on $\FF^p$, the Berezin transform of $T$ is the function defined by
$$\tilde{T}(z)=\ip{Tk_z}{k_z}_{\mathcal{F}^2}.$$ It was proved recently by Bauer and the first author that the vanishing
of the Berezin transform is sufficient for compactness whenever the operator is in the Toeplitz algebra \cite{BI}. However, it
is generally very difficult to check whether a given operator $T$ is in the Toeplitz algebra, unless $T$ is itself a
Toeplitz operator or a combination of a few Toeplitz operators, and as such one would like a ``simpler'' sufficient
condition to guarantee this.

In the recent and interesting paper \cite{XZ}, Xia and Zheng introduced a class of ``sufficiently localized'' operators on
$\FF^2$ which includes the algebraic closure of the Toeplitz operators. These are the operators $T$ acting on $\FF^2$ such that there exist constants $2n<\beta<\infty$ and $0<C<\infty$ with \begin{equation} \label{SL-Fock}
\abs{\ip{Tk_z}{k_w}_{\mathcal{F}^2}}\leq\frac{C}{\left(1+\abs{z-w}\right)^{\beta}}. \end{equation} It was proved by Xia
and Zheng that every bounded operator $T$ from the $C^*$ algebra generated by sufficiently localized operators whose
Berezin transform vanishes at infinity, i.e., \begin{equation}\label{Ber} \lim_{\abs{z}\to
\infty}\ip{Tk_z}{k_z}_{\mathcal{F}^2}=0 \end{equation} is compact on $\mathcal{F}^2$. One of their main innovations is
providing an easily checkable condition~\eqref{SL-Fock} which is general enough to imply compactness from the seemingly
much weaker condition~\eqref{Ber}.

The aim of this paper is threefold.  First, we wish to extend the Xia-Zheng notion of sufficiently localized operators to both a much wider class of weighted Fock spaces (in particular, the class of so-called ``generalized Bargmann-Fock spaces" considered in \cite{SV}) and to a larger class of operators.  Note that \eqref{SL-Fock} easily implies  $$
\sup_{z\in\C^n}\int_{\C^n}\abs{\ip{Tk_z}{k_w}_{\mathcal{F}^2}} \,dv(w)<\infty; $$ and consequently one should look at
generalizations of sufficiently localized operators that allow for weaker integral conditions. Also, note that the ideas in \cite{XZ} are essentially frame theoretic (see \cite{I} for a discussion of the ideas in \cite{XZ} from this point of view) and therefore one can not easily extend these ideas to the non-Hilbert space setting.  To remedy this, we will provide a simpler,
more direct proof of the main result in \cite{XZ} which follows a more traditional route and which can be extended to other
(not necessarily Hilbert) spaces of analytic functions.  In particular, we show that our main result, in an appropriately modified form, holds for the classical Bergman space $A^p$ on the ball (and in Section \ref{ConcRemSec} we will discuss the possibility of extending our results to a very wide class of weighted Bergman spaces.)

The extension of the main results in \cite{XZ} to a larger class of operators and to a wider class of weighted Fock spaces is as follows. Let $d^c =  \frac{i}{4} (\overline{\partial} - \partial)$ and let $d$ be the usual exterior derivative. For the rest of the paper let $\phi \in C^2(\C^n)$ be a real valued function on $\C^n$  such that \begin{equation*} c \omega_0 < d d^c \phi < C \omega_0 \end{equation*} holds uniformly pointwise on $\C^n$ for some positive constants $c$ and $C$ (in the sense of positive $(1, 1)$ forms) where $\omega_0 = d d^c |\cdot |^2$ is the standard Euclidean K\"{a}hler form.   Furthermore, for $0 < p \leq \infty$, define the generalized Bargmann-Fock space $\Fpphi$ to be the space of entire functions $f$  on $\C^n$ such that $fe^{-\phi} \in L^p(\C^n, dv)$ (for a detailed study of the linear space properties of $\Fpphi$ see \cite{SV}).   For operators $T$ acting on the reproducing kernels $K(z, w)$ of $\Ftwophi$, we
impose the following conditions.  We first assume that \begin{equation}\label{assump1-Fock}
\sup_{z\in\mathbb{C}^n}\int_{\mathbb{C}^n}\abs{\ip{Tk_z}{k_w}_{\Ftwophi}}\,dv(w)<\infty, \hspace{.5cm}
\sup_{z\in\C^n}\int_{\C^n}\abs{\ip{T^*k_z}{k_w}_{\Ftwophi}}\,dv(w)<\infty, \end{equation} which is enough to
conclude that the operator $T$ initially defined on the linear span of the reproducing kernels extends to a bounded
operator on $\Fpphi$ for $1 \leq p \leq \infty$ (see Section \ref{Fock}).  To show that the operator is compact, we impose the following additional assumptions on
$T$: \begin{equation}\label{assump-Fock}
 \lim_{r\to\infty}\sup_{z\in\C^n}\int_{D(z,r)^c}\abs{\ip{Tk_z}{k_w}_{\Ftwophi}}\,dv(w)=0, \hspace{.5cm}
 \lim_{r\to\infty}\sup_{z\in\C^n}\int_{D(z,r)^c}\abs{\ip{T^*k_z}{k_w}_{\Ftwophi}}\,dv(w)=0.
\end{equation} \begin{definition} \label{sufficient_local} We will say that a linear operator $T$ on $\Fpphi$ is weakly
localized (and for convenience write $T \in \AaFock$) if it satisfies the conditions~\eqref{assump1-Fock} and~\eqref{assump-Fock}. \end{definition} Note that every
sufficiently localized operator on $\mathcal{F}^2$ in the sense of Xia and Zheng obviously satisfies~\eqref{assump1-Fock}
and~\eqref{assump-Fock} and is therefore weakly localized in our sense too. Now if $D(z, r)$ is the Euclidean ball with center $z$ and radius $r$, and if $\|T\|_{\text{e}}$ denotes the essential norm of a bounded operator $T$ on $\Fpphi$ then the following theorem is one of the main results of this paper:

\begin{thm} \label{local-Fock} Let $ 1 < p < \infty$ and let $T$ be an operator on $\Fpphi$ which belongs to the norm closure of $\AaFock$. Then there exists $r, C > 0$ (both depending on $T$) such that \begin{equation*} \|T\|_{\text{e}} \leq C \limsup_{|z| \rightarrow \infty} \sup_{w \in D(z, r)} \abs{\ip{Tk_z}{k_w}}. \end{equation*}
In particular, if \begin{equation*} \lim_{|z| \rightarrow \infty} \|Tk_z\|_{\Fpphi} = 0 \end{equation*} then $T$ is compact on $\Fpphi$.  \end{thm} \noindent

Now if $\AaordFock$ is the class of sufficiently localized operators on $\Ft$ then note that an application of Proposition $1.4$ in \cite{I} in conjunction with Theorem \ref{local-Fock} immediately proves the following theorem, which provides the previously mentioned generalization of the results in \cite{XZ} (see Section \ref{Fock} for more details).

\begin{thm} \label{local-ordinaryFock} Let $ 1 < p < \infty$ and let $T$ be an operator on $\Fp$ which belongs to the norm closure of $\AaordFock$.  If $\lim_{|z| \rightarrow \infty} \abs{\ip{Tk_z}{k_z}_{\Ft}} = 0$ then $T$ is compact.  \end{thm}

Let us note that one can easily write the so called ``Fock-Sobolev spaces" from \cite{CZ} as generalized Bargmann-Fock spaces, so that in particular Theorem \ref{local-Fock} immediately applies to these spaces (see \cite{I} for more details).

To state the main result in the Bergman space setting requires some notation.  Let $\B_n$ denote the unit ball in $\C^n$
and let the space $A^p:=A^p(\B_n)$ denote the classical Bergman space, i.e., the collection of all holomorphic functions
on $\B_n$ such that $$ \norm{f}_{A^p}^p:=\int_{\B_n}\abs{f(z)}^p\,dv(z)<\infty. $$ The function
$K_z(w):=(1-\overline{z}w)^{-(n+1)}$ is the reproducing kernel for $A^2$ and $$
k_z(w):=\frac{(1-\abs{z}^2)^{\frac{n+1}{2}}}{(1-\overline{z}w)^{(n+1)}} $$ is the normalized reproducing kernel at the
point $z$.  We also will let $d\lambda$ denote the invariant measure on $\B_n$, i.e., $$
d\lambda(z)=\frac{dv(z)}{(1-\abs{z}^2)^{n+1}}. $$

Now let $1 < p < \infty$ and let $\frac1p + \frac{1}{p'} = 1$.   We are interested in operators $T$ acting on the reproducing kernels of $A^2$ that satisfy the following conditions.  First, we assume
that there exists $0 < \delta < \min\{p, p'\}$  such that \begin{equation}\label{assump1}
\sup_{z\in\B_n}\int_{\B_n}\abs{\ip{Tk_z}{k_w}_{A^2}}\frac{\norm{K_z}^{1 - \frac{2\delta}{p'(n + 1)}} _{A^2}}{\norm{K_w}^{1 - \frac{2\delta}{p'(n + 1)}} _{A^2}}\,d\la(w)<\infty,
\hspace{.5cm}
\sup_{z\in\B_n}\int_{\B_n}\abs{\ip{T^*k_z}{k_w}_{A^2}}\frac{\norm{K_z}^{1 - \frac{2\delta}{p(n + 1)}} _{A^2}}{\norm{K_w}^{1 - \frac{2\delta}{p(n + 1)}}  _{A^2}}\,d\la(w)<\infty.
\end{equation} \noindent These are enough to conclude that the operator $T$ initially defined on the linear span of the
reproducing kernels extends to a bounded operator on $A^p$ (see the comments following the proof of Proposition \ref{MainEst1}). To treat compactness we make the following additional
assumptions on $T$:  there exists $0 < \delta < \min\{p, p'\}$  such that \begin{equation}\label{assump}
 \sup_{z\in\B_n}\int_{D(z,r)^c}\abs{\ip{Tk_z}{k_w}_{A^2}}\frac{\norm{K_z}^{1 - \frac{2\delta}{p'(n + 1)}}  _{A^2}}{\norm{K_w}^{1 - \frac{2\delta}{p'(n + 1)}} _{A^2}}\,d\la(w) \rightarrow 0,
 \hspace{.5cm}
\sup_{z\in\B_n}\int_{D(z,r)^c}\abs{\ip{T^*k_z}{k_w}}\frac{\norm{K_z}^{1 - \frac{2\delta}{p(n + 1)}} _{A^2}}{\norm{K_w}^{1 - \frac{2\delta}{p(n + 1)}} _{A^2}}\,d\la(w) \rightarrow 0
\end{equation} as $r \rightarrow \infty$.

\begin{definition} \label{sufficient_local_Bergman} We say that a linear operator $T$ on $A^p$ is $p$ weakly localized (which we denote by $T \in \AapBerg$) if it
satisfies conditions~\eqref{assump1} and~\eqref{assump}. \end{definition} Note that the condition $0 < \delta < \min\{p, p'\}$ implies that both $1 - \frac{2\delta}{p(n + 1)}$ and $1 - \frac{2\delta}{p'(n + 1)}$ are strictly between $\frac{n - 1}{n + 1}$ and $1$.  Furthermore, note that when $p = p' = 2$, we have that   $ \frac{n - 1}{n + 1} < 1 - \frac{\delta}{(n + 1)} < 1$ precisely when $0 < \delta < 2$.  Thus, in this case we can rewrite condition ~\eqref{assump1} in the following simpler way: there exists $\frac{n-1}{n + 1} < a < 1$ where  \begin{equation*}
\sup_{z\in\B_n}\int_{\B_n}\abs{\ip{Tk_z}{k_w}_{A^2}}\frac{\norm{K_z}^{a} _{A^2}}{\norm{K_w}^{a} _{A^2}}\,d\la(w)<\infty,
\hspace{.5cm}
\sup_{z\in\B_n}\int_{\B_n}\abs{\ip{T^*k_z}{k_w}_{A^2}}\frac{\norm{K_z}^{a} _{A^2}}{\norm{K_w}^{a}  _{A^2}}\,d\la(w)<\infty. \end{equation*}
\noindent Of course, one can similarly rewrite condition ~\eqref{assump} when $p = 2$.

We prove the following result.

\begin{thm} \label{local-Bergman1} Let $1 < p < \infty$ and let $T$ be an operator on $A^p$ which belongs to the norm closure of $\AapBerg$. If \begin{equation*} \lim_{\abs{z}\to 1}\ip{Tk_z}{k_z}_{A^2}=0 \end{equation*}  then $T$ is compact. \end{thm} 

It will be clear that the
method of proof also will work for the weighted Bergman space $A^p _\alpha$, and we leave this to the interested reader to
verify.

Note that this result is known through deep work of Su\'arez, \cite{Sua} in the case of $A^p$ when the operator $T$ belongs to the Toeplitz algebra generated by $L^\infty$ symbols (see also \cite{MSW} for the case of
weighted Bergman spaces.) We will prove below that the  Toeplitz algebra on $A^p$ generated by $L^\infty$ symbols is a subalgebra of the norm closure of $\AapBerg$.   In particular, the results of this paper provide a considerably simpler proof of the main results in \cite{MSW, Sua} for the $p \neq 2$ situation (though it should be noted that a similar simplification when $p = 2$ was provided in \cite{MW}).

The structure of this paper is as follows.  In Section \ref{Bergman} we provide the extension of the the Xia and Zheng
result to the Bergman space on the unit ball $\B_n$, and in particular we prove Theorem~\ref{local-Bergman1}.  In Section \ref{Fock} we prove Theorems \ref{local-Fock} and \ref{local-ordinaryFock} which provides an extension of the Xia
and Zheng result in the case of the generalized Bargmann-Fock spaces. Finally in Section \ref{ConcRemSec} we will briefly discuss some interesting open problems related to these results.

\section{Bergman Space Case} \label{Bergman}

Let $\varphi_z$ be the M\"obius map of $\B_n$ that interchanges $0$ and $z$.  It is well known that $$ 1-\abs{\varphi_z(w)}^2=\frac{(1-\abs{z}^2)(1-\abs{w}^2)}{\abs{1-\overline{z}w}^{2}},  $$ and as a consequence we have that
\begin{equation} \label{Magic} \abs{\ip{k_z}{k_w}_{A^2}}=\frac{1}{\norm{K_{\varphi_z(w)}}_{A^2}}. \end{equation}

Using the automorphism $\varphi_z$, the pseudohyperbolic and Bergman metrics on $\B_n$ are defined by $$
\rho(z,w):=\abs{\varphi_z(w)}\quad\textnormal{ and }\quad \beta(z,w):=\frac{1}{2}\log\frac{1+\rho(z,w)}{1-\rho(z,w)}. $$
Recall that these metrics are connected by $\rho=\frac{e^{2\beta}-1}{e^{2\beta}+1}=\tanh\beta$ and it is well-known that
these metrics are invariant under the automorphism group of $\B_n$.  We let $$ D(z,r):=\{w\in\B_n:\beta(z,w)\leq
r\}=\{w\in\B_n: \rho(z,w)\leq s=\tanh r\}, $$ denote the hyperbolic disc centered at $z$ of radius $r$.
Recall also that the orthogonal (Bergman) projection of $L^2(\B_n, dv)$ onto $A^2$ is given by the integral operator $$
P(f)(z):=\int_{\B_n}\ip{\Kbw}{\Kbz}_{A^2}f(w)dv(w). $$ Therefore, for all $f\in A^2$ we have
\begin{equation} \label{BergResOfId} f(z)=\int_{\B_n}\ip{f}{\kbw}_{A^2}\kbw(z)\,d\lambda(w).\end{equation}

As usual an important ingredient in our treatment will be the Rudin-Forelli estimates, see \cite{Zhu} or \cite{MW}.
Recall the standard Rudin-Forelli estimates: \begin{equation} \label{propA6} \int_{\mathbb{B}_n}
{\frac{\abs{\ip{K_z}{K_w}_{A^2}}^{\frac{r+s}{2}}}{\norm{K_z}^s_{A^2}\norm{K_w}^r_{A^2}}\,d\la(w)}\leq C = C(r,s) <
\infty, \; \; \forall z \in \mathbb{B}_n \end{equation} for all $r>\kappa>s>0$, where $\kappa=\kappa_n:=\frac{2n}{n+1}$.
We will use these in the following form:  For all $\frac{n-1}{n+1}<a<1$ we have that
\begin{equation}\label{rf.n} \int_{\mathbb{B}_n}
\abs{\ip{k_z}{k_w}_{A^2}}\frac{\norm{K_z}^a_{A^2}}{\norm{K_w}^a_{A^2}}\,d\la(w)\leq C = C(a) < \infty, \; \; \forall z
\in \mathbb{B}_n. \end{equation} To see that this is true in the classical Bergman space setting, for a given
$\frac{n-1}{n+1}<a<1$ set $r=1+a$ and $s=1-a > 0$.   Then $r+s=2$, and since $a>\frac{n-1}{n+1}$ we have that
$r=1+a>\frac{2n}{n+1}$. Furthermore since $0 < a < 1$ we have that $0 < s < 1 \leq \frac{2n}{n + 1}$. By plugging this in \eqref{propA6} we obtain~\eqref{rf.n}.

We will also need the following uniform version of the Rudin-Forelli estimates.

\begin{lm} \label{lm-rfloc} Let $\frac{n-1}{n+1}<a<1$. Then \begin{equation}\label{rfloc}
\lim_{R\to\infty}\sup_{z\in\mathbb{B}_n}\int_{D(z,R)^c}
\abs{\ip{k_z}{k_w}_{A^2}}\frac{\norm{K_z}^a_{A^2}}{\norm{K_w}^a_{A^2}}\,d\la(w)= 0. \end{equation} \end{lm}

\begin{proof} Notice first that \begin{eqnarray*} \int_{D(z,R)^c}
\abs{\ip{k_z}{k_w}_{A^2}}\frac{\norm{K_z}^a_{A^2}}{\norm{K_w}^a_{A^2}}\,d\la(w)&=&\int_{D(0,R)^c}
\abs{\ip{k_z}{k_{\vp_z(w)}}_{A^2}}\frac{\norm{K_z}^a_{A^2}}{\norm{K_{\vp_z(w)}}^a_{A^2}}\,d\la(w)\\ &=& \int_{D(0,R)^c}
\abs{\ip{k_z}{k_w}_{A^2}}^a\frac{\norm{K_z}^a_{A^2}}{\norm{K_w}_{A^2}}\,d\la(w)\\ &=& \int_{D(0,R)^c}
\frac{\abs{\ip{K_z}{K_w}_{A^2}}^a}{\norm{K_w}^{1+a}_{A^2}}\,d\la(w)\\ &=& \int_{D(0,R)^c}
\frac{dv(w)}{\abs{1-\bar{w}z}^{(n+1)a}(1-\abs{w}^2)^{\frac{n+1}{2}(1-a)}}\\
&=&\int_{R'}^{1}\int_{\mathbb{S}_n}\frac{r^{2n-1}d\xi dr}{\abs{1-zr\overline{\xi}}^{(n+1)a}(1-r^2)^{\frac{n+1}{2}(1-a)}}
\end{eqnarray*} \noindent where in the last integral $R=\log\frac{1+R'}{1-R'}$. Notice that $R'\to1$ when $R\to\infty$ and note that the last integral can be written as $$\int_{R'}^{1}I_{(n+1)a-n}(rz)\frac{r^{2n-1} dr}{(1-r^2)^{\frac{n+1}{2}(1-a)}},$$
where $$I_{c}(z):=\int_{\mathbb{S}_n}\frac{d\xi}{\abs{1-zr\overline{\xi}}^{c+n}}.$$   By standard estimates (see \cite{Zhu}*{p. 15} for example), we have that
$$I_{(n+1)a-n}(rz)\lesssim
\begin{cases}
1, \hspace{.5cm} \mbox{if } \hspace{.2cm} (n+1)a-n<0 \\
\log\frac{1}{1-|rz|^2}, \hspace{.5cm} \text{if } \hspace{.2cm} (n+1)a-n=0 \\
\frac{1}{(1-|rz|^2)^{(n+1)a-n}}, \hspace{.5cm} \text{if } \hspace{.2cm} (n+1)a-n>0,

\end{cases}$$
which gives us that

$$ \int_{D(z,R)^c}
\abs{\ip{k_z}{k_w}_{A^2}}\frac{\norm{K_z}^a_{A^2}}{\norm{K_w}^a_{A^2}}\,d\la(w)\lesssim  \begin{cases}
\int_{R'}^1 \frac{r^{2n-1}}{(1-r^2)^{\frac{n+1}{2}(1-a)}}dr, \hspace{.5cm} \mbox{if } \hspace{.2cm} (n+1)a-n<0 \\
\int_{R'}^1 \log\frac{1}{1-r^2}\frac{r^{2n-1}}{(1-r^2)^{\frac12}}dr\, \hspace{.5cm} \text{if } \hspace{.2cm} (n+1)a-n=0 \\
\int_{R'}^1 \frac{r^{2n-1}}{(1-r^2)^{(n+1)a-n+\frac{n+1}{2}(1-a)}}dr, \hspace{.5cm} \text{if } \hspace{.2cm} (n+1)a-n>0

\end{cases}$$

Since $a < 1$, it is easy to see that all the functions appearing on the right hand side are integrable on $(0,1)$. Therefore, we obtain the desired conclusion by taking the limit as $R\to \infty$ (which is the same as $R'\to 1$).


\end{proof}



First, we want to make sure that the class of weakly localized operators is large enough to contain some interesting
operators. This is indeed true since every Toeplitz operator with a bounded symbol belongs to this class.

\begin{prop}\label{T-Berg} Each Toeplitz operator $T_u$ on $A^p$ with a bounded symbol $u(z)$ is in $\AapBerg$ for any $1 < p < \infty$.
\end{prop} \begin{proof} Clearly it is enough to show that \begin{equation*}
 \sup_{z\in\B_n}\int_{D(z,r)^c}\abs{\ip{T_u k_z}{k_w}_{A^2}}\frac{\norm{K_z}^{a}  _{A^2}}{\norm{K_w}^{a} _{A^2}}\,d\la(w) \rightarrow 0,
 \hspace{.5cm}
\sup_{z\in\B_n}\int_{D(z,r)^c}\abs{\ip{T_{\overline{u}} k_z}{k_w}}\frac{\norm{K_z}^{a} _{A^2}}{\norm{K_w}^{a} _{A^2}}\,d\la(w) \rightarrow 0
\end{equation*} as $r \rightarrow \infty$ for all $\frac{n - 1}{n + 1} < a < \infty$.

 By definition $$ T_uk_z(w)=P(uk_z)(w)=\int_{\B_n}\ip{K_x}{K_w}_{A^2}u(x)k_z(x)\,dv(x). $$
Therefore, \begin{eqnarray*} \abs{\ip{T_uk_z}{k_w}_{A^2}} & \leq &
\int_{\B_n}\abs{\ip{k_w}{k_x}_{A^2}}\abs{u(x)}\abs{\ip{k_z}{k_x}_{A^2}}\,d\la(x)\\
 & \leq & \left\Vert u\right\Vert_{\infty}\int_{\B_n}\abs{\ip{k_w}{k_x}_{A^2}\ip{k_x}{k_z}_{A^2}}\,d\la(x).
\end{eqnarray*} Now for $z,x\in\B_n$, set $$
I_z(x):= \abs{\ip{k_x}{k_z}_{A^2}} \int_{D(z,r)^c}\abs{\ip{k_w}{k_x}_{A^2}}\frac{\norm{K_z}^a_{A^2}}{\norm{K_w}^a_{A^2}}\,d\la(w)
$$ First note that \begin{eqnarray*}
\int_{D(z,r)^c}\abs{\ip{T_uk_z}{k_w}_{A^2}}\frac{\norm{K_z}^a_{A^2}}{\norm{K_w}^a_{A^2}}\,d\la(w) &\leq& \left\Vert
u\right\Vert_{\infty}
\int_{D(z,r)^c}\int_{\B_n}\abs{\ip{k_w}{k_x}_{A^2}\ip{k_x}{k_z}_{A^2}}\,d\la(x)\frac{\norm{K_z}^a_{A^2}}{\norm{K_w}^a_{A^2}}\,d\la(w)\\
&=& \left\Vert u\right\Vert_{\infty}
\int_{\B_n}\int_{D(z,r)^c}\abs{\ip{k_w}{k_x}_{A^2}}\frac{\norm{K_z}^a_{A^2}}{\norm{K_w}^a_{A^2}}\,d\la(w)\abs{\ip{k_x}{k_z}_{A^2}}\,d\la(x)\\
& = & \left\Vert u\right\Vert_{\infty} \int_{\B_n} I_z(x) \,d\la(x)\\ &=& \left\Vert
u\right\Vert_{\infty}\left(\int_{D(z,\frac{r}{2})}+\int_{D\left(z,\frac{r}{2}\right)^c}\right)I_z(x)\,d\lambda(x).
\end{eqnarray*} To estimate the first integral notice that for $x\in D\left(z,\frac{r}{2}\right)$ we have
$D(z,r)^c\subset D\left(x,\frac{r}{2}\right)^c$. Therefore, the first integral is no greater than

$$
\int_{D(z,\frac{r}{2})}\int_{D(x,\frac{r}{2})^c}\abs{\ip{k_w}{k_x}_{A^2}}\frac{\norm{K_z}^a_{A^2}}{\norm{K_w}^a_{A^2}}\,d\la(w)\abs{\ip{k_x}{k_z}_{A^2}}\,d\la(x).$$
It is easy to see that the last expression is no greater than $C(a)
A\left(\frac{r}{2}\right)$, where $$A(r)=\sup_{z\in\B_n}\int_{D(z,r)^c}
\abs{\ip{k_z}{k_w}_{A^2}}\frac{\norm{K_z}^a_{A^2}}{\norm{K_w}^a_{A^2}}\,d\la(w),$$ and $C(a)$ is just the bound from the
standard Rudin-Forelli estimates~\eqref{rf.n}.

Estimating the second integral is simpler. The second integral is clearly no greater than

$$
\int_{D\left(z,\frac{r}{2}\right)^c}\int_{\B_n}\abs{\ip{k_w}{k_x}_{A^2}}\frac{\norm{K_z}^a_{A^2}}{\norm{K_w}^a_{A^2}}\,d\la(w)\abs{\ip{k_x}{k_z}_{A^2}}\,d\la(x).$$
By the standard Rudin-Forelli estimates~\eqref{rf.n} the inner integral is no greater than $$
C(a)\frac{\norm{K_z}^a_{A^2}}{\norm{K_x}^a_{A^2}},$$ where the constant $C(a)$ is independent of $z$ and $x$. So, the
whole integral is bounded by $C(a)A\left(\frac{r}{2}\right)$. Therefore

$$\sup_{z\in\B_n} \int_{D(z,r)^c}\abs{\ip{T_uk_z}{k_w}_{A^2}}\frac{\norm{K_z}^a}{\norm{K_w}^a}\,d\la(w)\leq \left\Vert
u\right\Vert_{\infty}\left(C(a)A\left(\frac{r}{2}\right)+C(a)A\left(\frac{r}{2}\right)\right).$$ Applying the uniform
Rudin-Forelli estimates~\eqref{rfloc} in Lemma \ref{lm-rfloc} completes the proof since $2C(a)\left\Vert
u\right\Vert_{\infty}A\left(\frac{r}{2}\right)\to 0$ as $r\to\infty$. \end{proof}

We next show that the class of weakly localized operators forms a $*$-algebra.

\begin{prop}\label{C*-Berg} If $1 < p < \infty$ then $\AapBerg$ is an algebra.  Furthermore, $\AatwoBerg$ is a $*-$algebra. \end{prop}

\begin{proof} It is trivial that $T\in \AatwoBerg$ implies $T^*\in\AatwoBerg$. It is also easy to see that any linear combination of
two operators in $\AapBerg$ must be also in $\AapBerg$. It remains to prove that if $T, S\in \AapBerg$, then $TS\in \AapBerg$. To that end, we have that
\begin{align*}
\int_{D(z,r)^c} & \abs{\ip{TSk_z}{k_w}_{A^2}}\frac{\norm{K_z}^{1 - \frac{2\delta}{p'(n + 1)}} _{A^2}}{\norm{K_w}^{1 - \frac{2\delta}{p'(n + 1)}} _{A^2}}\,d\la(w) \\ & = \int_{D(z,r)^c}\abs{\ip{Sk_z}{T^*k_w}_{A^2}}\frac{\norm{K_z}^{1 - \frac{2\delta}{p'(n + 1)}}  _{A^2}}{\norm{K_w}^{1 - \frac{2\delta}{p'(n + 1)}}  _{A^2}}\,d\la(w)\\
\\ &=
\int_{D(z,r)^c}\abs{\int_{\B_n}\ip{Sk_z}{k_x}_{A^2}\ip{k_x}{T^*k_w}_{A^2}\,d\la(x)}\frac{\norm{K_z}^{1 - \frac{2\delta}{p'(n + 1)}}  _{A^2}}{\norm{K_w}^{1 - \frac{2\delta}{p'(n + 1)}}  _{A^2}}\,d\la(w)\\
&\leq
\int_{\B_n}\int_{D(z,r)^c}\abs{\ip{k_x}{T^*k_w}_{A^2}}\frac{d\la(w)}{\norm{K_w}^{1 - \frac{2\delta}{p'(n + 1)}}  _{A^2}}\abs{\ip{Sk_z}{k_x}_{A^2}}\norm{K_z}^{1 - \frac{2\delta}{p'(n + 1)}}
\,d\la(x).\\ \end{align*} Proceeding exactly as in the proof of the previous Proposition and using the conditions
following from $T, S\in \AapBerg$ in the place of the local Rudin-Forelli estimates~\eqref{rfloc} (and replacing $a$ with ${1 - \frac{2\delta}{p(n + 1)}})$ we obtain that $$
\lim_{r\to\infty}\sup_{z\in\B_n}\int_{D(z,r)^c}\abs{\ip{TSk_z}{k_w}_{A^2}}\frac{\norm{K_z}^{1 - \frac{2\delta}{p(n + 1)}} _{A^2}}{\norm{K_w}^{1 - \frac{2\delta}{p(n + 1)}} _{A^2}}\,d\la(w)=0.
$$ The corresponding condition for $(TS)^*$ is proved in exactly the same way. \end{proof}

We next show that every weakly localized operator can be approximated by infinite sums of well localized pieces. To state
this property we need to recall the following proposition proved in \cite{MW} \begin{prop} \label{Covering_Bergman} There
exists an integer $N>0$ such that for any $r>0$ there is a covering $\FF_r=\{F_j\}$ of $\B_n$ by disjoint Borel sets
satisfying \begin{enumerate} \item[\label{Finite} \textnormal{(1)}] every point of $\B_n$ belongs to at most $N$ of the
sets $G_j:=\{z\in\B_n: d(z, F_j)\leq r\}$, \item[\label{Diameter} \textnormal{(2)}] $\textnormal{diam}_d\, F_j \leq 2r$
for every $j$. \end{enumerate} \end{prop} We use this to prove the following proposition, which is similar to what
appears in \cite{MW}, but exploits condition \eqref{assump}.

\begin{prop}\label{MainEst1} Let $1 < p < \infty$ and let $T$ be in the norm closure of $\AapBerg$.    Then for every $\epsilon > 0$ there
exists $r>0$ such that for the covering $\FF_r=\{F_j\}$ (associated to $r$) from Proposition~\ref{Covering_Bergman}, we have:
\begin{eqnarray*} \norm{ TP-\sum_{j}M_{1_{F_j} }TPM_{1_{G_j} }}_{A^p\to L^p(\B_n, dv) } < \epsilon. \end{eqnarray*}
\end{prop}

\begin{proof} By Proposition~\ref{C*-Berg} in conjunction with Proposition \ref{Covering_Bergman} and a simple approximation argument, we may assume that $T \in \AapBerg$. Now define $$ S=TP-\sum_{j}M_{1_{F_j} }TPM_{1_{G_j} }.$$ Given $\epsilon$ choose $r$ large enough so
that \begin{equation*} \sup_{z\in\B_n}\int_{D(z,r)^c} \abs{ \ip{Tk_z}{k_w}_{A^2}}
\frac{\norm{K_z}^{1 - {\frac{2\delta}{p' (n + 1) }}  }_{A^2}}{\norm{K_w}^{1 - {\frac{2\delta}{p' (n + 1) }}  }_{A^2}}\,d\lambda(w)<\epsilon  \end{equation*} and \begin{equation*} \sup_{z\in\B_n}\int_{D(z,r)^c} \abs{ \ip{T^*k_z}{k_w}_{A^2}}\frac{\norm{K_z}^{1 - \frac{2\delta}{p (n + 1) }}    _{A^2}}{\norm{K_w}^{1 - \frac{2\delta}{p (n + 1) }}    _{A^2}}
\,d\lambda(w)<\epsilon. \end{equation*} Now for any $z\in\B_n$ let $z \in F_{j_0}$, so that \begin{eqnarray*} \abs{Sf(z)} & \leq &
\int_{\B_n}\sum_{j}1_{F_j(z)}1_{G_j^c}(w) \abs{ \ip{T^*\Kbz}{\Kbw}_{A^2} }\abs{f(w)} \,dv(w)\\
 & = & \int_{G_{j_0} ^c} \abs{ \ip{T^*\Kbz}{\Kbw}_{A^2} }\abs{f(w)} \,dv(w)\\
 & \leq & \int_{D(z,r)^c} \abs{ \ip{T^*\Kbz}{\Kbw}_{A^2} }\abs{f(w)} \,dv(w).
\end{eqnarray*}

To finish the proof, we will estimate the operator norm of the integral operator on $L^p(\B_n, dv)$ with kernel $1_{D(z, r) ^c} (w)|\langle T^*K_z, K_w\rangle_{A^2}| $ by using the classical Schur test.  To that end, let $h(w) = \|K_w\|_{A^2} ^{\frac{2\delta}{p p '(n + 1)}}$ so that \begin{align*} \int_{\B_n} 1_{D(z, r) ^c} (w)|\langle T^* K_z, K_w\rangle_{A^2}| h(w) ^{p'} \, dv(w) & = \int_{D(z, r)^c} |\langle T^* K_z, K_w\rangle_{A^2}|   \|K_w\|_{A^2} ^{\frac{2\delta}{p (n + 1) }}  \, dv(w) \\ & = \int_{D(z, r)^c} |\langle T^* k_z, k_w\rangle_{A^2}|   \|K_z\| \|K_w\|_{A^2} ^{\frac{2\delta}{p(n + 1) } - 1 }   \, d\lambda(w) \\ & \leq \epsilon  \|K_z\|_{A^2} ^{\frac{2\delta}{p(n + 1)}} = \epsilon h(z)  ^{p'}. \end{align*}

Similarly, we have that \begin{equation*} \int_{\B_n} 1_{D(z, r) ^c} (w)|\langle T^* K_z, K_w\rangle_{A^2}| h(z) ^p \, dv(z) \leq  \epsilon h(w) ^p \end{equation*} which completes the proof.

\end{proof}

It should be noted that a very similar Schur test argument actually proves that condition ~\eqref{assump1} implies that $T$ is bounded on $A^p$.

We can now prove one of our main results whose proof uses the ideas in ~\cite{MW}*{Theorem 4.3} and \cite{I}*{Lemma 5.3}. First, for any $w \in \B_n$ and $1 < p < \infty$,  let $k_w ^{(p)}$ be the ``$p$ - normalized reproducing kernel" defined by \begin{equation*}k_w ^{(p)} (z) = \frac{K(z, w) }{\|K_w\|^\frac{2}{p'} }. \end{equation*} Clearly we have that $k_w ^{(2)} = k_w$ and an easy computation tells us that $\|k_w ^{(p)}\|_{A^p} \approx 1$ (where obviously we have equality when $p = 2$).

\begin{thm}\label{essBerg}
 Let $1 < p < \infty$ and let $T$ be in the norm closure of $\AapBerg$.  Then there exists $r, C > 0$ (both depending on $T$) such that \begin{equation*} \|T\|_{\text{e}} \leq C \limsup_{|z| \rightarrow 1^{-}} \sup_{w \in D(z, r)} |\langle Tk_z ^{(p)} , k_w ^{(p')} \rangle_{A^2}| \end{equation*} where $\|T\|_{\text{e}}$ is the essential norm of $T$ as a bounded operator on $A^p$.
 \end{thm}

\begin{proof}Since $P : L^p(\B_n, dv) \rightarrow A^p$ is a bounded projection, it is enough to estimate the essential norm of $T = TP$
as an operator on from  $A^p$ to $L^p(\B_n, dv)$.

Clearly if $\|TP\|_\text{e} = 0$ then there is nothing to prove, so assume that $\|TP\|_\text{e} > 0.$ By Proposition~\ref{MainEst1} there exists $r>0$ such that for the covering $\FF_r=\{F_j\}$ associated
to $r$ (from Proposition~\ref{Covering_Bergman}) \begin{eqnarray*} \norm{TP- \sum_{j}M_{1_{F_j} }TPM_{1_{G_j} }}_{A^p\to
L^p(\B_n, dv)} < \frac{1}{2} \|TP\|_\text{e}. \end{eqnarray*}

Since $\sum_{j< m}M_{1_{F_j} }TPM_{1_{G_j} }$ is compact for every $m\in \N$ we have that $\|TP\|_{\text{e}}$ (as an operator from $A^p$ to $L^p(\B_n, dv)$) can be estimated in the following way:

\begin{align*} \norm{TP}_\text{e} &\leq  \norm{TP- \sum_{j < m}M_{1_{F_j} }TPM_{1_{G_j} }}_{A^p\to L^p(\B_n, dv)}\\ &\leq
\norm{TP- \sum_{j}M_{1_{F_j} }TPM_{1_{G_j} }}_{A^p\to L^p(\B_n, dv)}+\norm{T_m}_{A^p \rightarrow L^p(\B_n, dv)} \\ & \leq  \frac{1}{2} \|TP\|_\text{e} + \norm{T_m}_{A^p \rightarrow L^p(\B_n, dv)}, \end{align*} where
 \begin{equation*} T_m =  \sum_{j\geq m}M_{1_{F_j} }TPM_{1_{G_j} }. \end{equation*}  We will complete the proof by showing that there exists $C > 0$ where  \begin{equation*} \limsup_{m\to\infty}\norm{T_m}_{A^p\to
L^p(\B_n, dv)}\lesssim C \limsup_{|z| \rightarrow 1^{-}} \sup_{w \in D(z, r)} |\langle Tk_z ^{(p)}, k_w ^{(p')} \rangle_{A^2}| + \frac{1}{4}\norm{TP}_\text{e}. \end{equation*}

If $f\in A^p$ is arbitrary of norm no greater than $1$, then

\begin{align*} \norm{T_m f}_{A^p}^p &= \sum_{j\geq m}\norm{M_{1_{F_j} }TPM_{1_{G_j} }f}_{A^p}^p\\ &= \sum_{j\geq m}
\frac{\norm{M_{1_{F_j} }TPM_{1_{G_j} }f}_{A^p}^p}{\norm{M_{1_{G_j} }f}_{A^p}^p}\norm{M_{1_{G_j} }f}_{A^p}^p \leq
N\sup_{j\geq m}\norm{M_{1_{F_j} }Tl_j}_{A^p}^p \end{align*} where
$$l_j:=\frac{P M_{1_{G_j} }f}{\norm{M_{1_{G_j} }f}_{A^p}}.$$

Therefore, we have that

$$\norm{T_m}_{A^p\to L^p(\B_n, dv)}\lesssim \sup_{j\geq m}\sup_{\norm{f}_{A^p} \leq 1}\left\{\norm{M_{1_{F_j} } Tl_j}_{A^p}:
l_j=\frac{PM_{1_{G_j} }f}{\norm{M_{1_{G_j} }f}_{A^p}}\right\}$$ and hence

$$\limsup_{m\to \infty} \norm{T_m}_{A^p\to L^p(\B_n, dv)}\lesssim \limsup_{j\to\infty}\sup_{\norm{f}_{A^p}\leq
1}\left\{\norm{M_{1_{F_j} } Tl_j}_{A^p}: l_j=\frac{PM_{1_{G_j} }f}{\norm{M_{1_{G_j} }f}_{A^p}}\right\}.$$

Now pick a sequence $\{f_j\}$ in $A^p$ with $\norm{f_j}_{A^p}\leq 1$ such that
$$ \limsup_{j\to \infty}\sup_{\norm{f}\leq 1}\left\{\norm{M_{1_{F_j} } Tg}_{A^p}: g=\frac{PM_{1_{G_j}}f}{\norm{M_{1_{G_j}
}f}_{A^p}}\right\}-\frac{1}{4} \|TP\|_{\text{e}} \leq \limsup_{j\to\infty}\norm{M_{1_{F_j} } Tg_j}_{A^p},$$ where $$g_j=\frac{P M_{1_{G_j}
}f_j}{\norm{M_{1_{G_j}
}f_j}_{A^p}}= \frac{\int_{G_j} \ip{f}{ k_w ^{(p')}}_{A^2}  k_{w} ^{(p)} \,d\lambda(w)}{\left(\int_{G_j}\abs{\ip{f}{ k_u ^{(p')}  }_{A^2}}  ^p  \, d\lambda (u)\right)^{\frac{1}{p}}} = \int_{G_j} \widetilde{a}_j (w) \, k_w ^{(p)} \, d\lambda(w) $$ where \begin{equation*} \widetilde{a}_j (w) = \frac{\ip{f}{ k_w ^{(p')}}_{A^2}  }{\left(\int_{G_j}\abs{\ip{f}{ k_u ^{(p')}  }_{A^2}}  ^p  \, d\lambda (u)\right)^{\frac{1}{p}}}. \end{equation*}

Finally, by the reproducing property and H\"{o}lder's inequality, we have that \begin{align*} \limsup_{j \rightarrow \infty} \norm{M_{1_{F_j} } T g_j}_{A^p} ^p & \leq  \limsup_{j \rightarrow \infty} \int_{F_j} \left( \int_{G_j} \abs{\widetilde{a}_j (w) } \abs{Tk_w ^{(p)} (z)} \, d\lambda(w) \right)^p \, dv(z) \\ & = \limsup_{j \rightarrow \infty} \int_{F_j} \left( \int_{G_j} \abs{\widetilde{a}_j (w) } \abs{\ip{Tk_w ^{(p)}}{k_z ^{(p')}}_{A^2}} \, d\lambda(w) \right)^p \, d\lambda(z) \\ & \leq \limsup_{|z| \rightarrow 1^{-} } \sup_{w \in D(z, 3r)} \abs{\ip{Tk_z ^{(p)}}{ k_w ^{(p')}}_{A^2}}^p \left( \sup_j \lambda(G_j) ^p    \int_{G_j} \abs{\widetilde{a}_j (w) } ^p  \, d\lambda(w) \right) \\ & \leq C(r) \limsup_{|z| \rightarrow 1^{-} } \sup_{w \in D(z, 3r)} \abs{\ip{Tk_z ^{(p)}}{ k_w ^{(p')}}_{A^2}}^p \end{align*} since by Proposition \ref{Covering_Bergman} we have that $z \in F_j$ and $w \in G_j$ implies that $d(z, w) \leq 3r$ and $\lambda(G_j) \leq C(r)$ where $C(r)$ is independent of $j$.

 \end{proof}

We will finish this section off with a proof of Theorem \ref{local-Bergman1}.  First, for $z\in\B_n$,
define \begin{equation*}  U_z ^{(p)} f(w):= f(\varphi_z(w)) (k_z (w))^\frac{2}{p} \end{equation*} which via a simple change of variables argument is clearly an isometry on $A^p$.  As was shown in \cite{Sua}, an easy computation tells us that there exists a unimodular function $\Phi(\cdot, \cdot)$ on $\B_n \times \B_n$ where \begin{equation} \label{TransOpForm} (U_z ^{(p)})^* k_w ^{(p')}  = \Phi(z, w) k_{\phi_z (w)} ^{(p')}. \end{equation}

With the help of the operators $U_z ^{(p)}$, we will prove the following general result which in conjunction with Theorem \ref{essBerg} proves Theorem \ref{local-Bergman1}. Note that proof is similar to the proof of  \cite{I}*{Proposition 1.4}.

\begin{prop} \label{BerVanProp}  If $T$ is any bounded operator on $A^p$ for $1 < p < \infty$ then the following are equivalent \begin{itemize}{}{}
\item [(a)] $\lim_{|z| \rightarrow 1^{-}} \sup_{w \in D(z, r)} |\langle Tk_z ^{(p)} , k_w ^{(p')} \rangle_{A^2}|  = 0$ for all $r > 0$,
\item [(b)] $\lim_{|z| \rightarrow 1^{-}} \sup_{w \in D(z, r)} |\langle Tk_z ^{(p)} , k_w ^{(p')} \rangle_{A^2}|  = 0$ for some $r > 0$,
\item [(c)] $\lim_{|z| \rightarrow 1^{-}}  \abs{\ip{Tk_z}{k_z}_{A^2}} = 0 $. \end{itemize}
\end{prop}

\begin{proof} Trivially we have that $(a) \Rightarrow (b)$, and the fact that $(b) \Rightarrow (c)$ follows by definition and setting $z = w$.  We will complete the proof by showing that $(c) \Rightarrow (a)$.

 Assume to the contrary that $\abs{\ip{Tk_z}{k_z}_{A^2}}$ vanishes as $|z| \rightarrow 1^{-}$ but that \begin{equation*}  \limsup_{|z| \rightarrow 1^{-}} \sup_{w \in D(z, r)} \abs{\ip{Tk_z ^{(p)}}{ k_w ^{(p')}}_{A^2}} \neq 0 \end{equation*} for some fixed $r > 0$.  Thus, there exists sequences $\{z_m\}, \{w_m\} $ and some $0 < r_0 < 1$ where $\lim_{m \rightarrow \infty} |z_m| = 1$ and $|w_m| \leq r_0$ for any $m \in \N$, and where \begin{equation} \label{BerPropAssump} \limsup_{m \rightarrow \infty} \abs{\ip{ Tk_{z_m} ^{(p)}}{ k_{\varphi_{z_m} (w_m)} ^{(p')}}_{A^2}} > \epsilon\end{equation} for some $\epsilon > 0$.  Furthermore, passing to a subsequence if necessary, we may assume that $\lim_{m \rightarrow \infty} w_m = w \in \B_n$. Note that since $|w_m| \leq r_0 < 1$ for all $m$, we trivially have $\lim_{m \rightarrow \infty} k_{w_m} ^{(p')} =  k_w ^{(p')} $ where the convergence is in the $A^{p'}$ norm.

Let $\mathcal{B}(A^p)$ be the space of bounded operators on $A^p$.  Since the unit ball  in $\mathcal{B}(A^p)$ is $\WOT$ compact,  we can (passing to another subsequence if necessary) assume that \begin{equation*} \widehat{T} = \WOT -  \lim_{m \rightarrow \infty} U_{z_m} ^{(p)} T (U_{z_m} ^{(p')})^*. \end{equation*} Thus, we have that \begin{align*} \limsup_{m \rightarrow \infty} \abs{\ip{ Tk_{z_m}  ^{(p)}}{ k_{\varphi_{z_m} (w_m)}  ^{(p')}}_{A^2}} & =
\limsup_{m \rightarrow \infty} \abs{\ip{ U_{z_m} ^{(p)} T (U_{z_m} ^{(p')})^*  k_{0} ^{(p)} }{  k_{ w_m} ^{(p')}}_{A^2}} \\ & = \limsup_{m \rightarrow \infty} \abs{\ip{ U_{z_m} ^{(p)} T (U_{z_m} ^{(p')})^*  k_{0} ^{(p)} }{  k_{ w} ^{(p')}}_{A^2}} \\ & = \abs{\ip{\widehat{T} k_0}{ k_w}_{A^2}} .\end{align*}    However, for any $z \in \B_n$ \begin{equation*} \abs{\ip{\widehat{T} k_z ^{(p)}}{ k_z ^{(p')}}}  = \lim_{m \rightarrow \infty}  \abs{\ip{U_{z_m} ^{(p)} T (U_{z_m} ^{(p')})^* k_z ^{(p)}}{ k_z ^{(p')}} } \approx \lim_{m \rightarrow \infty} \abs{\ip{T k_{\varphi_{z_m} (z) } ^{(p)}}{ k_{\varphi_{z_m}(z)}  ^{(p')} }_{A^2}} = 0 \end{equation*} since by assumption $\abs{\ip{Tk_z}{k_z}}$ vanishes as $|z| \rightarrow {1^{-}}$. Thus, since the Berezin transform is injective on $A^p$, we get that $\widehat{T} = 0$, which contradicts (\ref{BerPropAssump}) and completes the proof. \end{proof}

\section{Generalized Bargmann-Fock space case} \label{Fock}

In this section we will prove Theorems \ref{local-Fock} and \ref{local-ordinaryFock}. Some parts of the proofs are essentially identical to proof of Theorem \ref{local-Bergman1} and so we will we only outline the necessary modifications.

For this section, let $$ D(z,r):=\left\{w\in\C^n:\abs{w-z}<r\right\} $$ denote the standard Euclidean disc centered at
the point $z$ of radius $r>0$.  For $z\in\C^n$, we define $$ U_z f(w):= f(z-w) k_z(w), $$ which via a simple change of
variables argument is clearly an isometry on $\Fp$ (though note in general that it is not clear whether $U_z$ even maps $\Fpphi$ into itself).     Recall also that the orthogonal projection of $L^2(\C^n,
e^{-2\phi} dv)$ onto $\Ftwophi$ is given by the integral operator $$
P(f)(z):=\int_{\C^n}\ip{\Kbw}{\Kbz}_{\Ftwophi}f(w)\,e^{-2\phi(w)} dv. $$ Therefore, for all $f\in \Fpphi$ we have \begin{equation} \label{FockResOfId}  f(z)=\int_{\C^n}\ip{f}{\widetilde{\kbw}}_{\Ftwophi}\widetilde{\kbw}(z)\,dv(w)\end{equation} where $\widetilde{\kbw}(z) := \Kbw(z) e^{-\phi(w)}$.  Note that $|K(z, z)| \approx e^{2\phi(z)}$ (see \cite{SV}) so that \begin{equation} \label{EquivOfNormRepKer} |\widetilde{\kbw}(z)| \approx |\kbw(z)|. \end{equation}

The following analog of Lemma~\ref{lm-rfloc} is simpler to prove in this case.

\begin{lm} \begin{equation}\label{rfloc-Fock} \lim_{R\to\infty}\sup_{z\in\mathbb{C}^n}\int_{D(z,R)^c}
\abs{\ip{k_z}{k_w}_{\Ftwophi}}\,dv(w)= 0. \end{equation} \end{lm} \noindent To prove this, simply note that there exists $\epsilon > 0$ such that $\abs{\ip{k_z}{k_w}_{\Ftwophi}} \leq e^{-\epsilon|z - w|}$ for all $z, w \in \C^n$.  The proof of this is then immediate since $$
\int_{D(z,R)^c} \abs{\ip{k_z}{k_w}_{\Ftwophi}}\,dv(w)\leq \int_{D(0,R)^{c}}  e^{-\epsilon|w|}  dv(w) $$ which clearly goes to
zero as $R\to\infty$.

As in the Bergman case, $\AaFock$ contains all Toeplitz operators with bounded symbols. Also, as was stated in the introduction, any $T \in \AaFock$ is automatically bounded on $\Fpphi$ for all $1 \leq p \leq \infty$.  To prove this, note that it is enough to prove that $T$ is bounded on $\Fonephi$ and $\Finfphi$ by complex interpolation (see \cite{I}).  To that end, we only prove that $T$ is bounded on $\Fonephi$ since the proof that $T$ is bounded on $\Finfphi$ is similar. If $T \in \AaFock$ and $f \in \Fonephi$, then the reproducing property gives us that \begin{align*} \abs{Tf(z)} e^{- \phi(z)} & \approx \abs {\ip{f}{T^*k_z}_{\Ftwophi}} \\ & \lesssim \int_{\C^n} \abs{f(u)} \abs{\ip{T^* k_z}{k_u}_{\Ftwophi}} \, e^{- \phi(u)} \, dv(u). \end{align*} Thus, by Fubini's theorem, we have that \begin{equation*} \norm{Tf}_{\Fonephi} \leq \int_{\C^n} \abs{f(u)} \left( \int_{\C^n}  \abs{\ip{T^* k_z}{k_u}_{\Ftwophi}} \,  dv(z) \right)  e^{- \phi(u)} \, dv(u) \lesssim \norm{f}_{\Fonephi}. \end{equation*}

In addition, $\AaFock$ satisfies the following two properties: \begin{prop} \label{T-Fock} Each Toeplitz operator $T_u$ on $\Fpphi$ with a bounded
symbol $u(z)$ is weakly localized. \end{prop}

\begin{proof} Since $\abs{\ip{k_z}{k_w}_{\Ftwophi}} \leq e^{-\epsilon|z - w|}$ for some $\epsilon > 0$ we have that \begin{align*} \abs{\ip{T_u k_z}{k_w}_{\Ftwophi}}  & \lesssim \|u\|_{L^\infty} \int_{\C^n} \abs{\ip{k_z}{k_x}_{\Ftwophi}} \abs{\ip{k_x}{k_w}_{\Ftwophi}} \, dx \\ & \lesssim \|u\|_{L^\infty} \int_{\C^n} e^{- \epsilon |z - x|} e^{-\epsilon |x - w |} \, dx. \end{align*}  Now if $|z - w| \geq r$ then by the triangle inequality we have that either $|z - x| \geq r/2$ or $|x - w| \geq r/2$ so that  \begin{equation*} \int_{D(z, r)^c} \abs{\ip{T_u {k}_z}{{k}_w}_{\Ftwophi}} \, dw  \lesssim e^{-\frac{\epsilon r}{2}} \|u\|_{L^\infty} \int_{D(z, r)^c} \int_{\C^n} e^{- \frac{\epsilon}{2}  |z - x|} e^{-\frac{\epsilon}{2} |x - w |} \, dx \,dw \lesssim  e^{-\frac{\epsilon r}{2}} \|u\|_{L^\infty}  \end{equation*} \end{proof}

Note that $T_u$ is sufficiently localized even in the sense of Xia and
Zheng by \cite{XZ}*{Proposition 4.1}.  Also note that a slight variation of the above argument shows that the Toeplitz operator $T_\mu \in \AaFock$ if $\mu$ is a positive Fock-Carleson measure on $\C^n$ (see \cite{SV} for precise definitions).

\begin{prop} \label{C*-Fock} $\AaFock$ forms a $*$-algebra.
\end{prop}  We will omit the proof of this proposition since it is proved in exactly the same way as it is in the Bergman space case (where the only difference is that one uses (\ref{FockResOfId}) in conjunction with (\ref{EquivOfNormRepKer}) instead of (\ref{BergResOfId})).

We next prove that operators in the norm closure of $\AaFock$ can also be approximated by infinite sums
of well localized pieces. To state this property we need to recall the following proposition proved in \cite{MW}
\begin{prop} \label{Covering} There exists an integer $N>0$ such that for any $r>0$ there is a covering $\FF_r=\{F_j\}$
of $\C^n$ by disjoint Borel sets satisfying \begin{enumerate} \item[\label{Finite} \textnormal{(1)}] every point of
$\C^n$ belongs to at most $N$ of the sets $G_j:=\{z\in\mathbb{C}^n: d(z, F_j)\leq r\}$, \item[\label{Diameter} \textnormal{(2)}]
$\textnormal{diam}_d\, F_j \leq 2r$ for every $j$. \end{enumerate} \end{prop} We use this to prove the following
proposition, which is similar to what appears in \cite{MW}, but exploits condition \eqref{assump-Fock} (and is proved in a manner that is similar to the proof of \cite{I}*{Lemma 5.2}). Note that for the rest of this paper, $\Lpphi$ will refer to the space of measurable functions $f$ on $\C^n$ such that $f e^{- \phi} \in L^p(\C^n, dv)$.

\begin{prop}\label{MainEst2} Let $1 < p < \infty$ and let $T$ be in the norm closure of $\AaFock$. Then
for every $\epsilon > 0$ there exists $r>0$ such that for the covering $\FF_r=\{F_j\}$ (associated to $r$) from
Proposition \ref{Covering}  \begin{eqnarray*} \norm{ TP-\sum_{j}M_{1_{F_j} }TPM_{1_{G_j} }}_{\Fpphi \to
\Lpphi } < \epsilon. \end{eqnarray*} \end{prop}

\begin{proof} Again by an easy approximation argument we can assume that $T \in \AaFock$. Furthermore, we first prove the theorem for $p = 2$.

 Define $$ S=TP-\sum_{j}M_{1_{F_j} }TPM_{1_{G_j} }. $$   Given $\epsilon$ choose
$r$ large enough so that \begin{equation*} \sup_{z\in\C^n}\int_{D(z,r)^c} \abs{\ip{T^*k_z}{k_w}_{\Ftwophi}}
dv(w)<\epsilon  \quad\textnormal{ and }\quad \sup_{z\in\C^n}\int_{D(z,r)^c} \abs{ \ip{Tk_z}{k_w}_{\Ftwophi}}
dv(w)<\epsilon. \end{equation*} Now for any $z\in\C^n$, pick $j_0$ such that $z \in F_{j_0}$.  Then we have that \begin{eqnarray*} \abs{Sf(z)} & \leq &
\int_{\C^n}\sum_{j}1_{F_j} (z)1_{G_j^c}(w) \abs{ \ip{T^*\Kbz}{\Kbw}_{\Ftwophi }}\abs{f(w)} e^{-2\phi(w)} \, dv(w)\\
 & = & \int_{G_{j_0} ^c} \abs{ \ip{T^*\Kbz}{\Kbw}_{\Ftwophi} }\abs{f(w)} e^{-2\phi(w)} \, dv(w) \\
 & \leq & \int_{D(z,r)^c} \abs{ \ip{T^*\Kbz}{\Kbw}_{\Ftwophi} }\abs{f(w)} e^{-2\phi(w)}\, dv(w).
\end{eqnarray*}

To finish the proof when $p = 2$, we will estimate the operator norm of the integral operator on $\Ltwophi$ with kernel $1_{D(z, r)^c} (w) \abs{ \ip{T^*\Kbz}{\Kbw}_{\Ftwophi} }$ using the classical Schur test. To that end, let $h(z) = e^{\frac{1}{2} \phi(z)}$ so that
\begin{equation*} \int_{\C^n} 1_{D(z, r)^c} (w) \abs{ \ip{T^*\Kbz}{\Kbw}_{\Ftwophi} }h(w)^2 e^{-2\phi(w)} \, dv(w) \approx h(z)^2  \int_{D(z, r)^c} \abs{\ip{T^* k_z}{k_w}_{\Ftwophi} } \, dv(w) \lesssim \epsilon h(z) ^2. \end{equation*}    Similarly, we have that \begin{equation*} \int_{\C^n} 1_{D(z, r)^c} (w) \abs{ \ip{T^*\Kbz}{\Kbw}_{\Ftwophi }}h(z)^2 e^{-2\phi(z)} \, dv(z)  \lesssim \epsilon h(w) ^2 \end{equation*} which finishes the proof when $p = 2$.

Now assume that $1 < p < 2$. Since $T$ is bounded on $\Fonephi$, we easily get that \begin{equation*} \norm{\sum_j M_{1_{F_j}} TP M_{1_{G_j}}}_{\Fonephi \rightarrow \Lonephi} < \infty \end{equation*} which by complex interpolation proves the proposition when $1 < p < 2$.  Finally when $2 < p < \infty$, one can similarly get a trivial $\Lonephi \rightarrow \Fonephi$ operator norm bound on \begin{equation*}  \left(\sum_j M_{1_{F_j}} TP M_{1_{G_j}}\right)^* = \sum_j P M_{1_{G_j}} T^* P M_{1_{F_j}} \end{equation*} since $T^*$ is bounded on $\Fonephi$.  Since $(\Fpphi)^* = \Fqphi$ when $1 < p < \infty$ where $q$ is the conjugate exponent of $p$ (see \cite{SV}), duality and complex interpolation now proves the proposition when $2 < p < \infty$.   \end{proof}

Because of (\ref{EquivOfNormRepKer}), the proof of the next result is basically the same as the proof of Theorem~\ref{essBerg} and therefore we skip it.

\begin{thm} \label{essFock} Let $1 < p < \infty$ and let $T$ be in the norm closure of $\AaFock$.  Then there exists $r, C > 0$ (both depending on $T$) such that \begin{equation*} \|T\|_{\text{e}} \leq C \limsup_{|z| \rightarrow\infty} \sup_{w \in D(z, r)} \abs{\ip{Tk_z}{ k_w }_{\Ftwophi}} \end{equation*} where $\|T\|_{\text{e}}$ is the essential norm of $T$ as a bounded operator on $\Fpphi$.\end{thm}

As was stated in the beginning of this section, the operator $U_z$ for $z \in \C^n$ is an isometry on $\Fp$.  Furthermore, since a direct calculation shows that \begin{equation*}\abs{ U_z k_w (u)} \approx \abs{k_{z - w} (u)},\end{equation*} the  proof of Theorem~\ref{local-ordinaryFock} now follows immediately by combining Theorem \ref{essFock} with \cite{I}*{Proposition 1.4}.

\section{Concluding remarks} \label{ConcRemSec}

The reader should clearly notice that the proof of Theorem \ref{essBerg} did \textit{not} in any way use the existence of a family of ``translation" operators $\{U_z ^{(p)} \}_{z \in \B_n}$ on $A^p$ that satisfies \begin{equation} \label{TransOpForm2} \abs{(U_z ^{(p)})^* k_w ^{(p')}}  \approx \abs{k_{\phi_z (w)} ^{(p')}} \end{equation}  (and moreover, one can make a similar remark regarding Theorem \ref{essFock}).  In particular, a trivial application of H\"{o}lder's inequality in conjunction with the above remark implies that one can prove the so called ``reproducing kernel thesis" for operators in the norm closure of $\AapBerg$ (respectively, $\AaFock$) \textit{without} the use of any ``translation" operators. It would therefore be interesting to know if our results can be proved for the weighted Bergman spaces on the ball that were considered in \cite{BO} for example.  Moreover, it would be interesting to know whether one can use the ideas in this paper to modify the results in \cite{MW} to include spaces where condition A.5 on the space of holomorphic functions at hand is not necessarily true (note that it is precisely this condition that allows one to easily cook up ``translation operators").

It would also be very interesting to know whether ``translation" operators are in fact crucial for proving Proposition \ref{BerVanProp} and its generalized Bargmann-Fock space analog (again see \cite{I}*{Proposition 1.4}).  More generally, it would be fascinating to know precisely how these translation operators fit into the ``Berezin transform implies compactness" philosophy since at present the answer to this seems rather mysterious.

As was noted earlier, the techniques in \cite{XZ} are essentially frame theoretic, and therefore are rather different than the techniques used in this paper.  In particular, a crucial aspect of \cite{XZ} involves a localization result somewhat similar in spirit to Proposition \ref{MainEst2} and which essentially involves treating a ``sufficiently localized" operator $T$ as a sort of matrix with respect to the frame $\{k_\sigma\}_{\sigma \in \Z^{2n}}$ for $\FF^2$. Also, note that the techniques in \cite{XZ} were extended in \cite{I} to the generalized Bargmann-Fock space setting to obtain results for $\Ftwophi$ that are similar to (but slightly weaker than) the results obtained in this paper.  Because of these considerable differences in localization schemes, it would be interesting to know if one can combine the localization ideas from this paper with that of \cite{I, XZ} to obtain new or sharper results on $\Ftwophi$ (or even just new or sharper results on $\Ft$).


\begin{bibdiv} \begin{biblist}


\bib{BI}{article}{
   author={Bauer, W.},
   author={Isralowitz, J.},
   title={Compactness characterization of operators in the Toeplitz algebra
   of the Fock space $F^p_\alpha$},
   journal={J. Funct. Anal.},
   volume={263},
   date={2012},
   number={5},
   pages={1323--1355}
}

\bib{BC}{article}{
   author={ Berger, C.},
   author={Coburn, L.},
   title={Heat flow and Berezin-Toeplitz estimates},
   journal={Amer. J. Math.},
   volume={116},
   date={1994},
   number={3},
   pages={563--590}}


\bib{BO}{article}{
   author={Berndtsson, B.},
   author={Ortega-Cerd\`{a}, J.},
   title={On interpolation and sampling in Hilbert spaces of analytic functions.},
   journal={J. Reine Angew. Math.},
   volume={464},
   date={1995},
   number={5},
   pages={109–-128}
}



\bib{CZ}{article}{
   author={Cho, H. R.},
   author={Zhu, K.},
   title={Fock-Sobolev spaces and their Carleson measures},
   journal={J. Funct. Anal.},
   volume={263},
   date={2012},
   number={8},
   pages={2483–-2506}
}




\bib{I}{article}{
author={Isralowitz, J.}
title={Compactness and essential norm properties of operators on generalized Fock spaces},
eprint={http://arxiv.org/abs/1305.7475},
status={to appear in J. Operator Theory},
date={2013}
pages={1--28}
}


\bib{MW}{article}{
   author={Mitkovski, M.},
   author={Wick, B. D.},
   title={A Reproducing Kernel Thesis for Operators on Bergman-type Function Spaces},
   journal={J. Funct. Anal.},
   volume={267},
   date={2014},
   pages={2028--2055}
}


\bib{MSW}{article}{
   author={Mitkovski, M.},
   author={Su{\'a}rez, D.},
   author={Wick, B. D.},
   title={The Essential Norm of Operators on $A^p_\alpha(\mathbb{B}_n)$},
   journal={Integral Equations Operator Theory},
   volume={75},
   date={2013},
   number={2},
   pages={197--233}
}
		

\bib{SV}{article}{
author={Schuster, A.}
author={Varolin, D.}
title={Toeplitz operators and Carleson measures on generalized Bargmann-Fock spaces,}
journal={Integral Equations Operator Theory}
volume={72}
date={2012}
number={3}
pages={363--392}
}

\bib{Sua}{article}{
   author={Su{\'a}rez, D.},
   title={The essential norm of operators in the Toeplitz algebra on $A^p(\mathbb{B}_n)$},
   journal={Indiana Univ. Math. J.},
   volume={56},
   date={2007},
   number={5},
   pages={2185--2232}
}


\bib{XZ}{article}{
   author={Xia, J.},
   author={Zheng, D.},
   title={Localization and Berezin transform on the Fock space},
   journal={J. Funct. Anal.},
   volume={264},
   date={2013},
   number={1},
   pages={97--117}
}

\bib{Zhu}{book}{
   author={Zhu, K.},
   title={Spaces of holomorphic functions in the unit ball},
   series={Graduate Texts in Mathematics},
   volume={226},
   publisher={Springer-Verlag},
   place={New York},
   date={2005},
   pages={x+271}
}

\bib{Zhu2}{book}{
    author={Zhu, K.}
    title={Analysis on Fock spaces},
    series={Graduate Texts in Mathematics},
    volume={263},
    publisher={Springer-Verlag},
    place={New York},
    date={2012},
    pages={x+344}
}

\end{biblist} \end{bibdiv}
\end{document}